\documentclass{article}
\usepackage{graphicx} 
\usepackage[english]{babel}
\usepackage{amsmath,amssymb,bbm,latexsym,tikz-cd,amsthm}

\title{On the Whitney disks in Heegaard Floer homology theory}
\author{Shengwen Xie and Xuezhi Zhao }

\newtheorem{theorem}{Theorem}
\newtheorem{prop}[theorem]{Proposition}
\newtheorem{remark}[theorem]{Remark}
\newtheorem{lemma}[theorem]{Lemma}
\newtheorem{definition}[theorem]{Definition}
\newtheorem{corollary}[theorem]{Corollary}

\newcommand\sym{\mathrm{Sym}}

\newcommand\cvp{\mathrm{CVP}}
\numberwithin{theorem}{section}

\begin{document}

\maketitle

\begin{abstract}
The Whitney disks play a central role in defining Heegaard Floer homology of a $3$-dimensional manifold. We use Nielsen theory to a simple criterion to  the existence of Whitney disks, connecting two given intersections.


\end{abstract}

\section{Introduction} 

Heegaard Floer homology is a collection of invariants for closed oriented
$3$-manifolds, introduced by Peter Ozsv\'{a}th and Zolt\'{a}n Szab\'{o} (\cite{spincc} and \cite{spincc-2}), and becomes a very powerful tool in the study of lower-dimensional topology. For example, it gives an alternate proof of the Donaldson's diagonalization theorem and the Thom conjecture for $CP^2$. It can be used to detect the Thurston norm of a $3$-manifold (\cite{Ni}, \cite{OZ-Thurston-norm}). Knot Floer homology detects the genus (\cite{OZ-knot-genus}) and
fiberedness (\cite{Ni-fibred-knot}) of knots and links in the $S^3$.

The computation of Heegaard Floer homology is highly non-trivial.  The generators of Heegaard Floer homology are the intersections of two tori in a symmetric product of two closed oriented surfaces, and the differentials count the number of points in certain moduli spaces of holomorphic disks connecting two intersection points, which are hard to compute in general. There are some attempts to deal with this computation, significative works can be found in \cite{Lipshitz} and \cite{Wang}. The main idea is to give  some combinatorial (or topological) descriptions of the chain complex of Heegaard Floer homology.

In this paper, we try to use generalized Nielsen fixed point theory to give a new treatment the chain complex of Heegaard Floer homology. We show that there exist a Whitney disk connecting two intersections $\mathbf{x}$ and $\mathbf{y}$ of totally real tori, which is a necessary condition to guarantee the existence of a differential between $\mathbf{x}$ and $\mathbf{y}$, if and only if the pre-image of these two intersections $\mathbf{x}$ and $\mathbf{y}$ lie in the same common value class. The common value classes are used to estimate the number of intersection homotopically (see \cite{Zhao}), and generalize the idea of Niesen fixed point theory (\cite{Brown} and \cite{Jiang}).

Thank to classical Nielsen fixed point theory, we develop a coordinate for each intersection, which takes value in the first homology of underlying $3$-manifold. Moreover, a practical method to compute these coordinates is given. A relation between fixed point and intersection was noticed in \cite{McCord1997}. Another different kind of relation between fixed point theory and Floer homology was introduced in \cite{R-Gautschi}. It our hope that the Nielsen fixed point theory may bring more to character Heegaard Floer homology. 

This paper is organized as follows: Section 2 contains a brief review of the definition of Heedaard Floer homology. In section 3, we extend the coordinates of fixed point classes into the case of common value classes. The correspondence of coordinates during a homotopy is also addressed. Section 4 devotes a new  
criterion to the existance of Whitney. The final section illustrates a way to compute the  coordinates of common value classes, and hence obtain a method to check if there is a Whitney disk between two given intersections. Moreover, the existence of a Whitney disk with Maslov index one is also discussed.

\section{A brief review of Heegaard Floer homology}

In this section, we will review some basic definitions and properties of Heegaard Floer homologies of $3$-manifolds. A brief introduction to the entire theory can be found in \cite{Whitneydisk}, while more details can be found in \cite{spinc}.

It is well-known that every closed, oriented $3$-dimensional manifold can be regarded as a union of two handlebodies. Such a decomposition of a $3$-dimensional manifold is called a Heegaard splitting.

Any Heegaard splitting of a $3$-dimensional manifold $M$ is described by a Heegaard diagram $(\Sigma_g, \mathbf{\alpha}, \mathbf{\beta})$, which consists of a closed oriented surface $\Sigma_g$ of genus $g$, along with two collections of disjoint simple loops $\mathbf{\alpha} = \{\alpha_1, \ldots, \alpha_g\}$ and $\mathbf{\beta}=\{\beta_1, \ldots, \beta_g\}$, satisfying the property that $\Sigma_g-\mathbf{\alpha}$ and $\Sigma_g -\mathbf{\beta}$ are both punctured spheres.

Given a Heegaard diagram $(\Sigma_g, \mathbf{\alpha}, \mathbf{\beta})$, the corresponding $3$-dimensional manifold is obtained as follows:
\begin{itemize}
    \item attaching a disk at each $\alpha_i$ on one side of $\Sigma_g$ for $i=1,2,\ldots, g$,
    \item attaching a disk at each $\beta_i$ on the other side of $\Sigma_g$ for $i=1,2,\ldots, g$,
    \item filling two $3$-balls. 
\end{itemize} 

The symmetric product $\sym_g(\Sigma_g)$ is defined as the orbit space $\Sigma_g^g / S_g$, where $S_g$ is the symmetric group on $g$ elements, and is actually a symplectic manifold. Given a Heegaard diagram $(\Sigma_g, \alpha_1,\ldots,\alpha_g, \beta_1,\ldots,\beta_g)$, since $\mathbf{\alpha}$ and $\mathbf{\beta}$ are both collections of simple loops, there are two $g$-dimensional tori 
$$T_\alpha:=\{(x_1,\ldots,x_g) \mid x_i \in \alpha_i\}/S_g,\ 
  T_\beta: =\{(y_1,\ldots,y_g) \mid y_i \in \beta_i\}/S_g,$$ 
called totally real tori in the symmetric space $\sym_g(\Sigma_g)$ (\cite{Fukaya}, \cite{Whitneydisk}). 

The intersection $T_\alpha\cap T_\beta$ of two totally real tori play a central role in the theory of Heedaard Floer homology.

\begin{definition}\label{disk} 
(\cite[Definition 5.1]{Whitneydisk})
Let $x, y \in T_\alpha\cap T_\beta$ be a pair of intersection points of the two tori. A Whitney disk connecting $x$ and $y$ is a map 
$(\mathbb{D}^2, e_1, e_2, -i, i ) \to (\sym_g(\Sigma_g),T_\alpha,T_{\beta}, x,y ),$ 
where $\mathbb{D}^2$ is the unit disk in the complex plane $\mathbb{C}$, $e_1, e_2$ are the arcs in the boundary
of $\mathbb{D}^2$ with $\mathrm{Re} \geq 0$ and $\mathrm{Re} \leq 0$.  
\end{definition}

Clearly, there is a Whitney disk connecting $x$ and $y$ if and only if there exist paths $\gamma': (I,0,1) \to (T_{\alpha}, x, y)$ and $\gamma'': (I, 0,1) \to (T_\beta, x, y)$ such that $\gamma' \simeq \gamma'': (I, 0,1) \to (\sym_g(\Sigma_g), x, y)$. In Heegaard Floer homology theory, we need to look for holomorphic representatives of the disks. The moduli space of holomorphic representatives of the disks $\phi$ is written as $\mathcal{M}(\phi)$.

\begin{definition}\label{def-HF}
Let $t\in \text{Spin}^C(Y)$ be a $\text{Spin}^C$ structure of a $3$-manifold $Y$, and $(\Sigma_g, \alpha_1,\ldots,\alpha_g ,\beta_1,\ldots,\beta_g, z)$ be a based Heegaard diagram of $Y$. The chain complex $\widehat{CF}(\alpha, \beta, t)$ is a free Abelian group generated by the points $x\in T_\alpha \cap T_\beta$ with $s_z(x) = t$, and the boundary homomorphism is given by
\[
\partial x =\sum_{y\in T_\alpha \cap T_\beta, \phi\in \pi_2(x,y)\mid 
s_z(y) = t, n_z(\phi)=0} c(\phi) y,
\]
where $c(\phi)$ is defined to be the sign number of the unparametrized moduli space $\hat{\mathcal{M}}(\phi)=\mathcal{M}(\phi)/\mathbb{R}$ if the Maslov index of $\mu(\phi)=1$; and $0$ otherwise.
\end{definition}

\section{Common value classes} 

We give a brief review of the theory of common value classes (see \cite{Zhao} for more details), which is a natural generalization of Nielsen fixed point theory \cite{Jiang}. After that we shall introduce a kind of coordinate for each common value class.

The concept of common value gives an approach to consider the intersection of two maps. Let $\varphi, \psi\colon X\to Y$ be two maps. Instead of thinking about the intersections  $\mathrm{Im}(\varphi)\cap \mathrm{Im}(\psi)$ in target space, we consider the pre-images at these intersections.

\begin{definition}(\cite[Definition 4.1]{Zhao})
Let $\varphi, \psi\colon X\to Y$ be two maps. The set of common value pairs $\cvp(\varphi,\psi)$ of $\varphi$ and $\psi$ is defined to be 
$$(\varphi\times \psi)^{-1}(\Delta_{Y^2})=\{(u, v)\in X^2 \mid \varphi(u) = \psi(v)\}.$$
\end{definition}

It is easy of check that the set $\cvp(\varphi, \psi)$ of common value pairs is a union of the projections of the sets of common value pairs of liftings of $\varphi$ and those of $\psi$ with respect to their universal coverings $p_X: \tilde X \to X$ and $p_Y: \tilde Y \to Y$.  By a regular arguments in Nielsen fixed point theory, for any two liftings $\tilde \varphi'$ and $\tilde \varphi''$ of $\varphi$ and any two two liftings $\tilde \psi'$ and $\tilde \psi''$ of $\psi$, the projection $p_X(\cvp(\tilde \varphi', \tilde \psi'))$ of the set $\cvp(\tilde \varphi', \tilde \psi')$ of common value pairs of $\tilde \varphi'$ and $\tilde \psi'$ and the projection $p_X(\cvp(\tilde \varphi'', \tilde \psi''))$ of the set $\cvp(\tilde \varphi'', \tilde \psi'')$ of common value pairs of $\tilde \varphi''$ and $\tilde \psi''$  is the same or distinct subset of $X^2$. Of course, the set of common value pairs may be empty, and the set $p_X(\cvp(\tilde \varphi, \tilde \psi))$ is said to be the common value class of $\varphi$ and $\psi$ determined by $(\tilde \varphi, \tilde \psi)$ (see \cite[Definition 4.3]{Zhao}). 

Fix a lifitng $\tilde \varphi_0$ of $\varphi$ and a lifting $\tilde \psi_0$ as references, each lifting of $\varphi$ can be written uniquely as $\gamma \tilde \varphi_0$ for some $\gamma$ in the deck trans formation group $D(\tilde Y)$ of $\tilde Y$,  while each lifting of $\psi$ can be written uniquely as $\delta \tilde \psi_0$ for some $\delta$ in the deck trans formation group $D(\tilde Y)$ of $\tilde Y$. Clearly, we have
$$\cvp(\varphi,\psi) = \cup_{\gamma\in D(\tilde Y)} p_X(\cvp (\tilde \varphi_0, \gamma \tilde \psi_0)).$$
Moreover, we have

\begin{prop}(\cite[Proposition 4.6]{Zhao})
Let $\tilde \varphi_0$ and $\tilde \psi_0$ be respectively liftings of $\varphi$ and $\psi$. Then
$(\tilde \varphi_0, \delta_1\tilde \psi_0)$ and $(\tilde \varphi_0, \delta_2\tilde \psi_0)$ determine the same common value class  if and only if $\delta_2 = \tilde f_{D}(\alpha^{-1})\delta_1\tilde \psi_{D}(\gamma)$ for some $\alpha, \gamma\in D(\tilde X)$, where $\tilde \varphi_{D}\colon D(\tilde X)\to D(\tilde Y)$ is the homomorphism determined by the relation $\tilde \varphi_0(\eta \tilde x) = \tilde \varphi_{D}(\eta)\tilde \varphi_0(\tilde x)$ for all $\eta\in D(\tilde X)$ and $\tilde x\in \tilde X$, and $\tilde \psi_{D}$ is defined similarly.
\end{prop}

By this proposition, each common value class of $\varphi$ and $\psi$ has a well-defined coordinate in the double coset $\varphi_D(D(\tilde X))\backslash D(\tilde Y)/\psi_D(D(\tilde X))$. It should be mentioned that some elements may not corresponds a ``real'' common value class, because $\cvp(\tilde \varphi_0, \tilde \gamma \psi_0)$ may be empty even if $\cvp(\varphi, \psi)\ne \emptyset$. 
Recall that any point in universal covering  can be regarded as a path class starting at the given base point. The left-action of deck transformation can be regarded as path product.

\begin{lemma}\label{cvp-coordinate}
Let $\varphi, \psi: X\to Y$ be two maps, $x_0$ and $y_0$ be respectively based points of $X$ and $Y$, and $\tilde \varphi_0$ and $\tilde \psi_0$ are respectively reference liftings of $\varphi$ and $\psi$, which are respectively determined by a path $w_\varphi$ from $y_0$ to $\varphi(x_0)$ and a path $w_\psi$ from $y_0$ to $\psi(x_0)$. Then the double coset coordinate in $\tilde \varphi_{0, D}\backslash D(\tilde Y)/\tilde \psi_{0, D}$ of common value pair $(u, v)$ of $\varphi$ and $\psi$ is given by $\langle w_\varphi \varphi(c_\varphi) \psi(c_\psi^{-1})w_\psi^{-1} \rangle$, where $c_\varphi$ is a path from $x_0$ to $u$ in $X$ and $c_\psi$ is a path from $x_0$ to $v$ in $X$.
\end{lemma}

\begin{proof}
Since $(u, v)$ is a common value pair  of $\varphi$ and $\psi$,
we have that $\tilde\varphi_0(\tilde u) = \gamma\tilde\psi_0(\tilde v)$ for some $\tilde u\in p_X^{-1}(u)$, $\tilde v\in p_X^{-1}(v)$, and $\gamma\in D(\tilde Y)$. Convert points of the universal covering space $\tilde Y$ into path classes starting at $y_0$, $\tilde u$ is regarded as a path $c_\varphi$ from $x_0$ to $u$ in $X$ and $\tilde v$ is regarded as a path $c_\psi$ from $x_0$ to $v$ in $X$. Two liftings are given by 
$$
\begin{array}{l}
 \tilde X \ni c: I, 0 \to X, x_0 \stackrel{\tilde \varphi_0}{\mapsto} w_\varphi \varphi(c):I, 0\to Y, y_0 \in \tilde Y, \\
 \tilde X \ni c: I, 0 \to X, x_0 \stackrel{\tilde \psi_0}{\mapsto} w_\psi \psi(c):I, 0\to Y, y_0 \in \tilde Y, 
\end{array}$$ 
and therefore $\tilde\varphi_0(\tilde u) = \gamma\tilde\psi_0(\tilde v)$ implies that $w_\varphi \varphi(c_\varphi) \dot{\simeq} \gamma w_\psi \psi(c_\psi)$. Here $\gamma$ is a loop at $y_0$, and hence lies in the same loop class as $w_\varphi \varphi(c_\varphi) \psi(c_\psi^{-1})w_\psi^{-1}$.
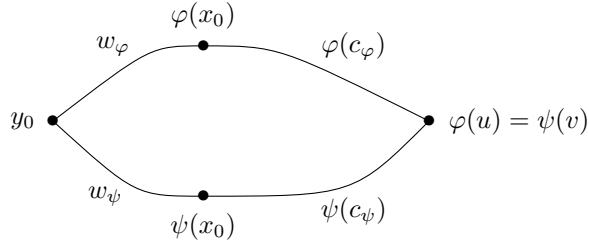
\begin{center}
\begin{tikzpicture}[scale=1]
\draw node at(2,-1) {$\bullet$};
\draw node at(2, 1) {$\bullet$};
\draw node at(0,0) {$\bullet$};
\draw node at(5,0) {$\bullet$};
\draw node at(6.2,0) {$\varphi(u)=\psi(v)$};
\draw node at(-.4,0) {$y_0$};
\draw node at(2, 1.4) {$\varphi(x_0)$};
\draw node at(2, -1.4) {$\psi(x_0)$};
\draw node at(4, -1.2) {$\psi(c_\psi)$};
\draw node at(4, 1) {$\varphi(c_\varphi)$};
\draw node at(0.7, -1) {$w_\psi$};
\draw node at(0.8, 1) {$w_\varphi$};

\draw (0,0) .. controls (1.3,1) .. (2,1);
\draw (0,0) .. controls (1.1,-1) .. (2,-1);
\draw (5,0) .. controls (3,1) .. (2,1);
\draw (5,0) .. controls (4,-1) .. (2,-1);
\end{tikzpicture}
\end{center}

If another point $\tilde u'\in p_X^{-1}(u)$ is chosen, we have another path $c'$  from $x_0$ to $u$. Thus,
the element
$$
\begin{array}{rcl}
 w_\varphi \varphi(c') \psi(c_\psi^{-1})w_\psi^{-1}
& = & w_\varphi \varphi(c'c^{-1}_\varphi) w^{-1}_\varphi w_\varphi \varphi(c_\varphi) \psi(c_\psi^{-1})w_\psi^{-1} \\
& = & \tilde \varphi_{0, D}(c'c^{-1}_\varphi)  w_\varphi \varphi(c_\varphi)\psi(c_\psi^{-1})w_\psi^{-1} 
\end{array}$$
is the same as $w_\varphi\varphi(c_\varphi) \psi(c_\psi^{-1})w_\psi^{-1}$ in double coset 
$$\tilde \varphi_{0, D}\backslash D(\tilde Y)/\tilde \psi_{0, D}=\tilde \varphi_{0, \pi}\backslash \pi_1(Y, y_0)/\tilde \psi_{0,\pi}.$$ 
Thus, the coordinate in this double coset is independent of the choice of path $c_\varphi$ initial at the base point. The proof of independency of the choice of the path $c_\psi$ is similar.
\end{proof}

By this lemma, each common value pair of $\varphi$ and $\psi$ has a well-defined  coordinate in double coset $\tilde \varphi_{0, D}\backslash D(\tilde Y)/\tilde \psi_{0, D} =\tilde \varphi_{0, \pi}\backslash \pi_1(Y, y_0)/\tilde \psi_{0,\pi}$, and two common value pairs has the same coordinates if and only if they are in the same common value class. The identification of two double cosets is given by
\[
\begin{tikzcd}
D(\tilde X) \arrow[r, "\tilde \varphi_{0, D}\  \tilde \psi_{0, D}"] \arrow[d, "\cong"]  & D(\tilde Y) \arrow[d, "\cong"] \\
\pi_1(X, x_0)
\arrow[r, "\tilde \varphi_{0, \pi}\ \tilde \psi_{0,\pi}"]
& \pi_1(Y,y_0).                    \end{tikzcd}
\]
The isomorphism is determined by $D(\tilde X)\ni\gamma \mapsto p_X \eta\in \pi_1(X,x_0)$, where $\eta$ is a path from $\tilde x_0$ (regarded as constant loop at $x_0$) to $\gamma(\tilde x_0)$. The induced homomorphism $\tilde \varphi_{0, \pi}: \pi_1(X,x_0)\to \pi_1(Y,y_0)$ is given by $\tilde \varphi_{0, \pi}(\langle c\rangle) = \langle  w_\varphi \varphi(c) w^{-1}_\varphi\rangle$.  The identification of $D(\tilde Y)$ $\pi_1(Y,y_0)$ and the homomorphism $\psi_{0, \pi}$ are given similarily. For convenience in the next part of this article we only use $\tilde \varphi_{0, \pi}\backslash \pi_1(Y, y_0)/\tilde \psi_{0,\pi}$.

If $\varphi:X\to Y$ is homotopic to $\varphi':X\to Y$ and $\psi:X\to Y$ is homotopic to $\psi':X\to Y$. Then any homotopy $H_\varphi:X\times I \to Y$ connecting $\varphi$ and $\varphi'$ gives to a one-to-one correspondence from the set of liftings of $\varphi$ to that of $\varphi'$, and any homotopy $H_\psi:X\times I \to Y$ connecting $\psi$ and $\psi'$ gives to a one-to-one correspondence from the set of liftings of $\psi$ to that of $\psi'$. 

In fact, the two correspondences of determined by two homotopies $H_\varphi$ and $H_\psi$ induce a correspondence from the set of common value classes of $\varphi$ and $\psi$ to the set of common value classes of $\varphi'$ and $\psi'$. 

\begin{definition}\label{def-Hrelated}
Let $H_\varphi: X\times I \to Y$ be a homotopy from  $\varphi:X\to Y$  to $\varphi':X\to Y$, and let $H_\psi: X\times I \to Y$ be a homotopy from  $\psi: X\to Y$  to $\psi':X\to Y$.  A common value class $p_X(\cvp(\tilde \varphi,\tilde \psi))$ of $\varphi$ and $\psi$ and a common value class $p_X(\cvp(\tilde \varphi',\tilde \psi'))$  of $\varphi'$ and $\psi'$  are said to be $(H_\varphi, H_\psi)$-related if $\tilde \varphi$ and  $\tilde \varphi'$  are respectively $0$- and $1$-slices of a lifting $\tilde H_\varphi:\tilde X\times I \to \tilde Y$ of $H_\varphi$, and $\tilde \psi$ and  $\tilde \psi'$  are respectively $0$- and $1$-slices of a lifting $\tilde H_\psi:\tilde X\times I \to \tilde Y$ of $H_\psi$.
\end{definition}
\begin{remark}
A common value class of $\varphi$ and $\psi$ is $(I_\varphi, I_\psi)$-related to itself where $I_\varphi(\_,t)=\varphi,I_\psi(\_,t)=\psi$.
\end{remark}
\begin{theorem}\label{eq H related}
     Let $H_\varphi: X\times I \to Y$ be a homotopy from  $\varphi:X\to Y$  to $\varphi':X\to Y$, and let $H_\psi: X\times I \to Y$ be a homotopy from  $\psi: X\to Y$  to $\psi':X\to Y$. A common value class $C$ of $\varphi$ and $\psi$ and a common value class $C'$ of $\varphi'$ and $\psi'$ are $(H_\varphi,H_\psi)$-related if and only if there exist $(u,v)\in C$ and $(u',v')\in C'$, a path $a$ from $u$ to $u'$ and path $b$ from $v$ to $v'$ such that $\{H_\varphi(a(t),t)\}_{t\in I}\dot{\simeq} \{H_\psi(b(t),t)\}_{t\in I}$ 
\end{theorem}
\begin{proof}
   only If:  A common value class $C$ of $\varphi$ and $\psi$ and a common value class $C'$ of $\varphi'$ and $\psi'$ are $(H_\varphi,H_\psi)$-related, so there exist liftings $\tilde{\varphi}, \tilde{\varphi'}, \tilde{\psi}, \tilde{\psi'}$ such that $C=p_X(\cvp(\tilde \varphi,\tilde \psi)), C'= p_X(\cvp(\tilde \varphi',\tilde \psi'))$ and $\tilde \varphi$ and  $\tilde \varphi'$  are respectively $0$- and $1$-slices of a lifting $\tilde H_\varphi:\tilde X\times I \to \tilde Y$ of $H_\varphi$, and $\tilde \psi$ and  $\tilde \psi'$  are respectively $0$- and $1$-slices of a lifting $\tilde H_\psi:\tilde X\times I \to \tilde Y$ of $H_\psi$.Thus there exist $(u,v)\in C$, $(u',v')\in C'$ and $(\tilde u,\tilde v)\in \cvp(\tilde \varphi,\tilde \psi)$ , $(\tilde u',\tilde v')\in \cvp(\tilde \varphi',\tilde \psi')$ such that $P_X(\tilde u,\tilde v)=(u,v)$, $P_X(\tilde u',\tilde v')=(u',v')$. 
   
   For the connectivity of $\tilde X $ there exist a path $\tilde a$ from $\tilde u$ to $\tilde u'$ and path $\tilde b$ from $\tilde v$ to $\tilde v'$ .The path   $\{\tilde H_\varphi(\tilde a(t),t)\}_{t\in I}$ 
 from $\tilde H_\varphi(\tilde a(0),0)=\tilde\varphi(\tilde u)$ to $\tilde H_\varphi(\tilde a(1),1)=\tilde \varphi'(\tilde v')$ and the path $\{\tilde H_\psi(\tilde b(t),t)\}_{t\in I}$ from $\tilde H_\psi(\tilde b(0),0)=\tilde \psi(\tilde u)$ to $\tilde H_\psi(\tilde b(1),1)=\tilde \psi'(\tilde u')$ have the same  starting point and end point . 
 
 For the universal space $\tilde Y $ is simply connected, we have  
 
 $$\{\tilde H_\varphi(\tilde a(t),t)\}_{t\in I}\dot{\simeq} \{\tilde H_\psi(\tilde b(t),t)\}_{t\in I}$$ which implies $\{ H_\varphi(P_X(\tilde a(t)),t)\}_{t\in I}\dot{\simeq} \{ H_\psi(P_X(\tilde b(t)),t)\}_{t\in I}$.Thus we find a path $a=P_X(\tilde a(t))$ from $u$ to $u'$ and path $b=P_X(\tilde b(t))$ from $v$ to $v'$ such that $\{H_\varphi(a(t),t)\}_{t\in I}\dot{\simeq} \{H_\psi(b(t),t)\}_{t\in I}$ 

 if: There exist $(u,v)\in C$ and $(u',v')\in C'$, a path $a$ from $u$ to $u'$ and path $b$ from $v$ to $v'$ such that $\{H_\varphi(a(t),t)\}_{t\in I}\dot{\simeq} \{H_\psi(b(t),t)\}_{t\in I}$ .Take  liftings of $\{H_\varphi(a(t),t)\}_{t\in I}$ ,$\{H_\psi(b(t),t)\}_{t\in I}$ denoted by $\tilde c, \tilde d$  determined by $\tilde c(0)=\tilde d(0)$ which implies $\tilde c(1)=\tilde d(1)$ for the universal space $\tilde Y $ is simply connected. Take $\tilde u \in P_X^{-1}(u), \tilde v \in P_X^{-1}(v)$  which determine the lifting  of $a, b$ with these two points as their starting point  denoted by $\tilde a, \tilde b$ .There is an unique lifting of  $\varphi$ denoted $\tilde \varphi$ such that $\tilde \varphi(\tilde u)=\tilde c(0)$ and there is an unique lifting of  $\psi$ denoted $\tilde \psi$ such that $\tilde \psi(\tilde u)=\tilde d(0)$. For the homotopy lifting property we can extend  $\tilde \varphi$ to $\tilde H_\varphi:\tilde X\times I \to \tilde Y$ and extend  $\tilde \psi$ to $\tilde H_\psi:\tilde X\times I \to \tilde Y$ .Denote $\tilde H_\varphi(\_, 1),\tilde H_\psi(\_, 1)$ as $\tilde \varphi', \tilde \psi'$ which are liftings of $\varphi',\psi'$ respectively.
 $$P_Y(\{\tilde H_\varphi(\tilde a(t),t)\}_{t\in I})=\{H_\varphi(a(t),t)\}_{t\in I},P_Y(\{\tilde H_\psi(\tilde b(t),t)\}_{t\in I})= \{H_\psi( b(t),t)\}_{t\in I}$$
  $\{\tilde H_\varphi(\tilde a(t),t)\}_{t\in I})=\tilde c(t)$ and $\{\tilde H_\psi(\tilde b(t),t)\}_{t\in I}=\tilde d(t)$ for they have the same staring point which implies $\tilde \varphi'(\tilde u')=\tilde \varphi'(\tilde v')$ i.e.  $(\tilde u,\tilde v)\in \cvp(\tilde \varphi,\tilde \psi)$ , $(\tilde u',\tilde v')\in \cvp(\tilde \varphi,\tilde \psi)$ such that $P_X(\tilde u,\tilde v)=(u,v)$, $P_X(\tilde u',\tilde v')=(u',v')$. The common value class $C$ of $\varphi$ and $\psi$ and the common value class $C'$ of $\varphi'$ and $\psi'$ are $(H_\varphi,H_\psi)$-related.
\end{proof}

Now we consider the relation of coordinates between homotopies.


\begin{theorem}
Let $H_\varphi: X\times I \to Y$ be a homotopy from  $\varphi$  to $\varphi'$, and let $H_\psi: X\times I \to Y$ be a homotopy from  $\psi$  to $\psi'$. Let $x_0$ be the basepoint of  $X$ and $y_0$ be the basepoint of  $Y$.  If the reference paths $w_\varphi: I, 0, 1\to Y, y_0, \varphi(x_0)$,$w_{\varphi'}: I, 0, 1\to Y, y_0, \varphi'(x_0)$, $w_\psi: I, 0, 1\to Y, y_0, \psi(x_0)$,$w_{\psi'}: I, 0, 1\to Y, y_0, \psi'(x_0)$  are chosen such that 
\begin{equation}\label{Hrelated-ref}
w_{\varphi'} = w_\varphi\ast \{H_\varphi(x_0, t)\}_{t\in[0,1]},\ \ w_{\psi'} = w_\psi 
\ast\{H_\psi(x_0,t)\}_{t\in[0,1]}.
\end{equation}
Then both maps $\tilde\varphi$ and $\tilde\varphi'$ induce the same homomorphism from $\pi_1(X, x_0)$ to $\pi_1(Y, y_0)$, and both maps $\tilde\psi$ and $\tilde\psi'$ induce the same homomorphism from $\pi_1(X, x_0)$ to $\pi_1(Y, y_0)$. Hence, $$\tilde \varphi_{\pi}\backslash \pi_1(Y, y_0)/\tilde \psi_{\pi}=\tilde \varphi'_{\pi}\backslash \pi_1(Y, y_0)/\tilde \psi'_{\pi}.$$
Moreover, a common value class of $\varphi$ and $\psi$, and a common value class of $\varphi'$ and $\psi'$ are $(H_\varphi, H_\psi)$-related  if and only if they have the same coordinates.   
\end{theorem}

\begin{proof}
For any $c\in \pi(X, x_0)$, we have that $$\tilde \psi'_{0, \pi}(c)=w_{\psi'} \psi'(c) (w_{\psi'})^{-1}= w_\psi \ast \{H_\psi(x_0,t)\}_{t\in[0,1]} \psi'(c) ( w_\psi \ast \{H_\psi(x_0,t)\}_{t\in[0,1]})^{-1}.$$ 
There is a homotopy between $w_{\psi'} \psi'(c) (w_{\psi'})^{-1}= w_\psi \ast \{H_\psi(x_0,t)\}_{t\in[0,1]} \psi'(c) ( w_\psi \ast \{H_\psi(x_0,t)\}_{t\in[0,1]})^{-1}$ and $w_\psi \psi(c) ( w_\psi )^{-1}$, which is given by 
$$I\times I\ni (t,s)\mapsto w_\psi \{H_\psi(x_0,st)\} \{H_\psi(c(t),s)\} (w_\psi \{H_\psi(x_0,st)\})^{-1}.$$
The equality $\tilde \varphi'_{0, \pi}(c)=\tilde \varphi_{0, \pi}(c)$ holds for the same reason.
  There exist $(u,v)\in C$ and $(u',v')\in C'$, a path $a$ from $u$ to $u'$ and path $b$ from $v$ to $v'$ such that $\{H_\varphi(a(t),t)\}_{t\in I}\dot{\simeq} \{H_\psi(b(t),t)\}_{t\in I}$ .Take $c_{\varphi'}=c_\varphi a,c_{\psi'}=c_\psi b$ the coordinate of $(u',v')$ is represented by 
$$
\begin{array}{lcl}

w_\varphi \{H_\varphi(x_0, t)\}_{t\in[0,1]} \varphi'(c_\varphi a) (w_\psi \{H_\psi(x_0, t) \}_{t\in[0, 1]}\psi'(c _\psi b))^{-1} \\
 \simeq w_\varphi  \varphi(c_\varphi) \{H_\varphi(u, t)\}_{t\in[0,1]}\varphi'(a)(w_\psi \psi(c _\psi )\{H_\psi(v, t) \}_{t\in[0, 1]}\psi' (b))^{-1} \\
 \simeq w_\varphi  \varphi(c_\varphi) \{H_\varphi(a(t),t)\}_{t\in I}(w_\psi \psi(c _\psi ) \{H_\psi(b(t),t)\}_{t\in I})^{-1}.\\
  
\end{array}$$

The first homotopy is given by
$$I\times I\ni (t,s)\mapsto \{H_\varphi(x_0,st)\} \{H_\varphi(c_\varphi(t),s)\} \{H_\varphi(u,(1-t)s+t)\} $$ $$I\times I\ni (t,s)\mapsto \{H_\psi(x_0,st)\} \{H_\psi(c_\psi(t),s)\} \{H_\psi(v,(1-t)s+t)\}$$

The second homotopy is given by
$$I\times I\ni (t,s)\mapsto \{H_\varphi(u,st)\} \{H_\varphi(a(st),s)\}\{H_\varphi(a((1-t)s+t),(1-t)s+t)\} $$

$$I\times I\ni (t,s)\mapsto \{H_\psi(u,st)\} \{H_\psi(b(st),s)\}\{H_\psi(b((1-t)s+t),(1-t)s+t)\} $$
Thus  $ w_\varphi  \varphi(c_\varphi) (w_\psi \psi(c _\psi ) )^{-1}\simeq w_{\varphi'}  \varphi'(c_{\varphi'}) (w_{\psi'} \psi(c _{\psi'} ) )^{-1}$ if and only if

$\{H_\varphi(a(t),t)\}_{t\in I}\dot{\simeq} \{H_\psi(b(t),t)\}_{t\in I}$ 
\end{proof}

\section{Whitney disks and common value classes}

In this section we are going to use common value class to consider Whitney disks connecting two intersection  $T_\alpha \cap  T_\beta$ of two totally real tori in the symmetric product space $\sym_g(\Sigma_g)$, giving a new perspective to understand the key object in Heegaard Floer homology.

Given a Heegaard diagram $(\Sigma_g, \bf{\alpha}, \bf{\beta})$, two sets of simple loops give two embedding:
$$\iota_\alpha,  \iota_\beta: T^g=(S^1)^g \to (\Sigma_g)^g$$
defined by
$$\iota_\alpha(\theta_1,\ldots,\theta_g)=(\alpha_1(\theta_1),\ldots,\alpha_g(\theta_g)), \  \iota_\beta(\theta_1,\ldots,\theta_g)=
(\beta_1(\theta_1),\ldots,\beta_g(\theta_g)).$$ 
Hence, there is a commutative diagram: 
\[
\begin{tikzcd}
 & (\Sigma_g)^g \arrow[d, "\pi"] \\
T^g=\arrow[ru, "{\iota_\alpha, \iota_\beta}"] \arrow[r, "{\Bar{\iota}_\alpha, \Bar{\iota}_\beta}"] 
& \sym_g(\Sigma_g),                              
\end{tikzcd}
\]
where $\pi$ is the quotient map from the product space $(\Sigma_g)^g$ to its corresponding symmetric product space.





In this setting, since $\iota_\alpha, \iota_\beta$ are both embedding, intersection $T_\alpha\cap T_{\beta}$ of two totally real tori is the same as  $(\bar{\iota}_\alpha\times \bar{\iota}_\beta) (\cvp (\bar{\iota}_\alpha,\bar{\iota}_\beta))$.  Of most importance is:

\begin{theorem} \label{eq-preclass-D}
Let $x, y\in T_\alpha\cap T_{\beta}$ be two intersection of two totally real tori. Then there is a Whitney disk connecting $x$ and $y$ if and only if $(\bar{\iota}^{-1}_\alpha(x), \bar{\iota}^{-1}_\beta(x))$ and  $(\bar{\iota}^{-1}_\alpha(y), \bar{\iota}^{-1}_\beta(y))$ are in the same common value class of  $\bar\iota_\alpha$ and $\bar \iota_\beta$. 
\end{theorem}

\begin{proof} Since the curves $\alpha_i$ are disjoint from each other $\alpha_j$ curves when $j$ is not equal to $i$ and the curves $\beta_i$ are disjoint from each other $\beta_j$ curves when $j$ is not equal to $i$, we have that if $\bar{x}=(\bar{x_1},\ldots, \bar{x_g})\in Im\iota_\alpha$ then $\sigma \bar{x}=(\bar{x}_{\sigma(1)},\ldots, \bar{x}_{\sigma(g)})\notin Im\iota_\alpha$. Thus $\pi | Im\iota_\alpha$ is an injective map and $\pi | Im\iota_\beta$ is an injective map with the same reason. It follows that $\bar{\iota}_\alpha$ and $\bar{\iota}_\beta$ are both embedding. For any $w\in T_\alpha\cap T_{\beta}$, both  $\bar{\iota}^{-1}_\alpha(w)$ and  $\bar{\iota}^{-1}_\beta(w)$ are well-defined points in $T^g$, and $(\bar{\iota}^{-1}_\alpha(w),\bar{\iota}^{-1}_\beta(w))$ is a common value pair of $\bar{\iota}_\alpha$ and $\bar{\iota}_\beta$.

If: Suppose that two intersections $x,y\in T_\alpha\cap T_{\beta}$ lie in the same common value class. Consider the universal covering $p:\tilde W\to \sym_g\Sigma_g$ of $\sym_g\Sigma_g$ and the universal covering $p_T: R^g \to T^g$ of $T^g$. By definition, we have that $$(\bar{\iota}^{-1}_\alpha(x),  \bar{\iota}^{-1}_\beta(x)), (\bar{\iota}^{-1}_\alpha(y),  \bar{\iota}^{-1}_\beta(y)) \in (p\times p)(({\tilde{\bar{\iota}}_\alpha \times  \tilde{\bar{\iota}}_\beta})^{-1}(\Delta(\tilde{Y}^2)))$$ 
for some lifting $\tilde{\bar{\iota}}_\alpha$ of $\bar{\iota}_\alpha$ and some lifting $\tilde{\bar{\iota}}_\beta$ of $\bar{\iota}_\beta$.
Thus there exist $\tilde{x}, \tilde{y} \in \tilde{W}$ and $\tilde{u}_1,\tilde{u}_2,\tilde{v}_1,\tilde{v}_2\in R^g$ such that 
\[\begin{tikzcd}
(R^g\times R^g, (\tilde{u}_1,\tilde{u}_2)) 
 \arrow[r, "\tilde{\bar{\iota}}_\alpha \times \tilde{\bar{\iota}}_\beta "] 
 \arrow[d, "p_T\times p_T"] 
  & (\tilde{W}\times\tilde{W}, (\tilde{x}, \tilde{x})) \arrow[d, "p\times p"] \\
(T^g\times T^g, (\bar{\iota}^{-1}_\alpha(x),  \bar{\iota}^{-1}_\beta(x))) \arrow[r, "\bar\iota_\alpha \times \bar\iota_\beta"] 
 & (\sym_g(\Sigma_g) \times \sym_g(\Sigma_g), (x,x\bar))                 \end{tikzcd}
\]
and that
\[\begin{tikzcd}
(R^g\times R^g, (\tilde{v}_1,\tilde{v}_2)) 
 \arrow[r, "\tilde{\bar{\iota}}_\alpha \times \tilde{\bar{\iota}}_\beta "] 
 \arrow[d, "p_T\times p_T"] 
  & (\tilde{W}\times\tilde{W}, (\tilde{y}, \tilde{y})) \arrow[d, "p\times p"] \\
(T^g\times T^g, (\bar{\iota}^{-1}_\alpha(y),  \bar{\iota}^{-1}_\beta(y))) 
  \arrow[r, "\bar\iota_\alpha \times  \bar\iota_\beta"] 
 & (\sym_g(\Sigma_g) \times \sym_g(\Sigma_g), (y,y))                       
\end{tikzcd}
\]

The connectivity of  $R^g$ implies that there is a path $(\tilde{a}(t), \tilde{b}(t))$ connecting $(\tilde{u}_1,\tilde{u}_2)$ and  $(\tilde{v}_1,\tilde{v}_2)$. 
Since $\tilde{\bar{\iota}}_{\alpha}(\tilde{a}(t)) $ and 
$\tilde{\bar{\iota}}_{\beta}(\tilde{b}(t)) $ have the same endpoints $\tilde x$ and $\tilde y$, and $\tilde{W}$ is simply connected, it follows that $\tilde{\bar{\iota}}_{\alpha}(\tilde{a}(t)) \dot{\simeq} \tilde{\bar{\iota}}_{\alpha}(\tilde{b}(t)) $. Then \[
\bar{\iota}_\alpha p_T\tilde{a}(t) =p\tilde{\bar{\iota}}_{\alpha}(\tilde{a}(t)) \dot{\simeq} p \tilde{\iota}_{\beta}(\tilde{b}(t)) = \bar{\iota}_\beta p_T\tilde{b}(t).
\]
Obviously, $\bar{\iota}_\alpha p_T\tilde{a}(t)$ is in $T_\alpha$ and $\bar{\iota}_\beta p_T\tilde{b}(t)$ is in $T_{\beta}$, and they have the same endpoints $x,y$. Hence, there is a Whitney disk connecting $x$ and $y$.

Only if: Suppose that there is a Whitney disk connecting $x$ and $y$, i.e. there exist paths $f(t): (I, 0,1) \to (T_{\alpha}, x, y)$ and $g(t):(I, 0,1) \to (T_\beta, x, y)$ such that $f(t) \dot{\simeq} g(t)$.

Take a $\tilde{x}\in p^{-1}(x)$ in the universal covering space $\tilde W$, and take a lifting $\tilde{f}$ of $f$ and a $\tilde{g}$ of $g$ such that $\tilde{f}(0)=\tilde{g}(0)=\tilde{x}$. Since $f(t) \dot{\simeq} g(t)$, we have $\tilde{f}(t) \simeq \tilde{g}(t)$.

The universal covering space $\tilde W$ is simply connected, which implies that $\tilde{f}(1)=\tilde{g}(1)$, and we denote the endpoint as $\tilde{y}$. Since $\bar{\iota}_\alpha$ and $\bar{\iota}_\beta$ are both embedding  $(\bar{\iota}^{-1}_\alpha(f(t)),\bar{\iota}^{-1}_\beta(g(t))$ is a well-defined curve in $T^g\times T^g$ with starting point $(\bar{\iota}^{-1}_\alpha(x),\bar{\iota}^{-1}_\beta(x))$ and with endpoint $(\bar{\iota}^{-1}_\alpha(y),\bar{\iota}^{-1}_\beta(y))$. Choose a lifting of path  $(\bar{\iota}^{-1}_\alpha(f(t)),\bar{\iota}^{-1}_\beta(g(t))$ denoted by $(F(t), G(t))$ determined by the starting point $(u, v)\in (p_T\times p_T)^{-1}(\bar{\iota}^{-1}_\alpha(x),\bar{\iota}^{-1}_\beta(x))$. Denote the lifting of $\bar{\iota}_\alpha \times \bar{\iota}_\beta$ as $\tilde{\bar{\iota}}_\alpha \times \tilde{\bar{\iota}}_\beta$ determined by mapping $(u,v)$ to $(\tilde{x}, \tilde{x})$. Note that
$$
\begin{array}{rcl}
(p\times p)(\tilde{\bar{\iota}}_\alpha \times \tilde{\bar{\iota}}_\beta )(F(t), G(t))
& = &(\bar{\iota}_\alpha \times \bar{\iota}_\beta)(p_T\times p_T)(F(t),G(t))\\
& = & (\bar{\iota}_\alpha \times \bar{\iota}_\beta)(\bar{\iota}^{-1}_\alpha(f(t)),\bar{\iota}^{-1}_\beta(g(t)) \\
& = & (f(t),g(t)).
\end{array}
$$
We have that $(\tilde{\bar{\iota}}_\alpha \times \tilde{\bar{\iota}}_\beta )(F(t), G(t))$ is a lifting of $(f(t),g(t))$ with the same starting point as $(\tilde{f}(t), \tilde{g}(t))$. Then  $(\tilde{\bar{\iota}}_\alpha \times \tilde{\bar{\iota}}_\beta )(F(t), G(t))=(\tilde{f}(t), \tilde{g}(t))$ for the uniqueness properties of the lifting.  Thus,  
$$
\begin{array}{rcl}
(\bar{\iota}^{-1}_\alpha(x),\bar{\iota}^{-1}_\beta(x))
& = & (p_T\times p_T)(F(0), G(0))\\
& = & (p_T\times p_T) (\tilde{\bar{\iota}}_\alpha \times \tilde{\bar{\iota}}_\beta )^{-1}(\tilde{f}(0), \tilde{g}(0)) \\
& = & (p_T\times p_T) (\tilde{\bar{\iota}}_\alpha \times \tilde{\bar{\iota}}_\beta )^{-1}(\tilde{x},\tilde{x})
\end{array}$$ and
$$
\begin{array}{rcl}
(\bar{\iota}^{-1}_\alpha(y),\bar{\iota}^{-1}_\beta(y))
& = &(p_T\times p_T)(F(1), G(1))\\
& = & (p_T\times p_T) (\tilde{\bar{\iota}}_\alpha \times \tilde{\bar{\iota}}_\beta )^{-1}(\tilde{f}(1), \tilde{g}(1))\\
& = & (p_T\times p_T) (\tilde{\bar{\iota}}_\alpha \times \tilde{\bar{\iota}}_\beta )^{-1}(\tilde{y},\tilde{y})\\
\end{array}$$ 
That is,  $(\bar{\iota}^{-1}_\alpha(x),\bar{\iota}^{-1}_\beta(x))$ and $(\bar{\iota}^{-1}_\alpha(y),\bar{\iota}^{-1}_\beta(y))$ are in the same common value class.
\end{proof}
\begin{corollary}\label{H_1(M)}
There is an injective map from the set of Whitney disks to $H_{1}(M)$ where $M$ is the three-manifold defined by the Heegaard diagram 

$(\Sigma_g, \alpha_1,\ldots,\alpha_g, \beta_1,\ldots,\beta_g)$.
    
\end{corollary}
 \begin{proof}
   The set of  common value class is $$\tilde{\bar{\iota}}_{\alpha,\pi}\backslash \pi_1(\sym_g\Sigma_g, y_0)/\tilde{\bar{\iota}}_{\beta,\pi}\cong \tilde{\bar{\iota}}_{\alpha,\pi}\backslash H_1(\Sigma_g)/\tilde{\bar{\iota}}_{\beta,\pi} \cong H_1(M)$$ , then by \ref{eq-preclass-D} there is an injective map from the set of Whitney disks to $H_{1}(M)$
 \end{proof}

\begin{corollary}
Let $M$ be the three-manifold defined by the Heegaard diagram $(\Sigma_g, \alpha_1,\ldots,\alpha_g, \beta_1,\ldots,\beta_g)$. If $H_{1}(M)=0$, then for any $x, y\in T_\alpha\cap T_{\beta}$, there is a Whitney disk connecting $x$ and $y$.
\end{corollary}

\section{Coordinates of intersections}

In this section, we shall show the way to get the coordinates of intersections of two totally real tori.

By Lemma~\ref{cvp-coordinate}, each intersection in $T_\alpha\cap T_\beta$ has a coordinate in the double coset $\mathrm{Im}\bar\iota_{\alpha,D} \backslash \pi_1(\sym_g(\Sigma_g))/\mathrm{Im}\bar\iota_{\beta,D}$, which is actually homology group $H_1(M)$ of given $3$-dimensional manifold by  corollary \ref{H_1(M)} .
\begin{remark}
The coordinate of the intersection point class is in $H_1(M)$, which is isomorphic to $H^2(M, \mathbb{Z})$ by Poincaré duality. In \cite{spinc}, a choice of $s_0\in Spin^C(M)$ gives an affine isomorphism between $Spin^C(M)$ and $H^2(M, \mathbb{Z})$. Thus, our coordinate provides a correspondence between the common value  class and $Spin^C(M)$.  It is established that for any $x, y \in T_\alpha \cap T_\beta$, a Whitney disk connects them if and only if their corresponding $Spin^C$ structure is identical. This enables the decomposition of the Heegaard Floer Homology group by $Spin^C$ structure.
\end{remark}

We use following presentation of fundamental group $\pi_1(\Sigma_g) $ of surface $\Sigma_g$ of genus $g$:
\begin{equation}\label{presentation}
\pi_1(F_g, y_0) = \langle c_1, c_2, \ldots, c_{2g}\mid c_1 c_2 \cdots c_{2g} = c_{2g} \cdots c_2 c_1\rangle,
\end{equation}

\begin{theorem}
Let $\alpha_i$ be a loop determine by a word $u_{i,1}u_{i,2}\cdots u_{i,m_i}$ for $i=1,2,\ldots, g$, and let $\beta_i$ be a loop determine by a cyclically $D$-reduced word $v_{j,1}v_{j,2}\cdots v_{j,n_j}$ for $j=1,2,\ldots, g$, where each $u_{*,*}$ and $v_{*,*}$ is a letter in the set $\{c^{\pm}_1, c^{\pm}_2, \ldots, c^{\pm}_{2g}\}$. Then each common value pair of $\bar\iota_\alpha: T^g\to \sym_g(\Sigma_g)$ and $\bar\iota_\beta: T^g\to \sym_g(\Sigma_g)$  \begin{equation}
\sum_{i=1}^g [u_{i,1}]+\cdots+[u_{i, k_i}] +[v_{\sigma(i), l_\sigma(i)+1}]+\cdots+[v_{\sigma(i), n_{\sigma(i)}}] \in H_1(M),
\end{equation}
where $\sigma$ is an element in the symmetric group $S_g$, and $1\le k_j\le m_i$, $1\le l_j \le n_j$ for all $i,j$.
\end{theorem} 

\begin{proof}
Recall from \cite[Theorem 2.7]{Zhao2022} that any common value class of $\alpha_i:S^1\to \Sigma_g$ and $\beta_j:S^1\to \Sigma_g$ has a coordinate $$u_{i,1}\cdots u_{i, k_i}v_{j,l_j+1}\cdots v_{j,n_j}\in \mathrm{Im} \alpha_i \backslash \pi_1(\Sigma_g)/\mathrm{Im} \beta_j.$$
Thus,the coordinates of the common value classes of $\iota_\alpha, \sigma\iota_\beta: T^g\to (\Sigma_g)^g$ are the products of elements of the forms above. We obtain immediately our conclusion by natural map from $(\Sigma_g)^g$ to $\sym_g\Sigma_g$ and its induced homomorphism, which is exactly an abelization.
\end{proof}

By this Theorem, one can decide easily if two intersections lie in the same class, i. e. there is a Whitney disk connecting them.


Suppose that there is a Whitney disk connecting two points 
$x,y\in T_\alpha\cap T_\beta$, i. e. 
$(\bar{\iota}^{-1}_\alpha(\mathbf{x}),\bar{\iota}^{-1}_\beta(\mathbf{x}))$ 
and 
$(\bar\iota^{-1}_\alpha(\mathbf{y}),\bar\iota^{-1}_\beta(\mathbf{y}))$ 
lies in the same common value class. By \cite[Proposition 2.15]{spincc}, if $g>1$, the homotopy class $\pi_2(\bf{x},\bf{y})$ of Whitney disks is the same as $\mathbb{Z}\oplus H^1(M)$. The first factor is $\pi_2(\sym_g(\Sigma_g))$ if $g>2$, is $\pi_2(\sym_g(\Sigma_g))$ modulo the action of $\pi_1(\sym_g(\Sigma_g))$ if $g=2$.

Although there are a lot of Whitney disks, by definition of Heegaard Floer homology (see Definition~\ref{def-HF}), only the disks with Maslov index one involved.

The surface $\Sigma_g$ can be regarded as a complex manifold, and therefore has a canonical orientation. Once each loop in Heegaard diagram is given an orientation, for example if each is given by an element of $\pi_1(\Sigma_g)$, then each intersection $\alpha_i\cap \beta_j$ has a well-defined intersection number, which is $\pm 1$ because of the transverality.

From \cite[Definition 3.1]{Zhao}, we know that the index of the pair $P=(\bar{\iota}^{-1}_\alpha(\mathbf{x}),\bar{\iota}^{-1}_\beta(\mathbf{x}))$ is a composition 
$$
\begin{array}{ll}
H_*((T^g)^2)&\to 
  H_*((T^g)^2, (T^g)^2 - P)\\
 & \stackrel{e^{-1}_*}{\to} 
  H_*(N(P), N(P)-P) \\
& \stackrel{(\bar{\iota}_\alpha\times \bar{\iota}_\beta)_*}{\to} 
 H_*((\sym_g \Sigma_g)^2, (\sym_g \Sigma_g)^2-\Delta).
\end{array}$$

It is obvious that

\begin{prop}
Let $\mathbf{x}=(x_1, x_2, \ldots, x_g) \in T_\alpha\cap T_\beta$ be an intersection , where  $x_i\in \alpha_i\cap \beta_{\sigma(i)}$ for some $\sigma\in S_g$. Then its intersection number $I(\bar{\iota}_\alpha, \bar{\iota}_\beta; \mathbf{x})$ is equal to $\mathrm{sgn}(\sigma)\prod I(\alpha_i, \beta_{\sigma(i)},x_i)$. Moreover, the homomorphism index $$\mathcal{L}=\mathcal{L}(\bar{\iota}_\alpha\times \bar{\iota}_\beta, (\bar{\iota}^{-1}_\alpha(\mathbf{x}),\bar{\iota}^{-1}_\beta(\mathbf{x})), \Delta)$$ of common value pair 
$(\bar{\iota}^{-1}_\alpha(\mathbf{x}),\bar{\iota}^{-1}_\beta(\mathbf{x}))$ is given by $\mathcal{L}([(T^g)^2]) = I(\bar{\iota}_\alpha, \bar{\iota}_\beta; \mathbf{x}) [\Delta] $, 
where $[(T^g)^2]$ is the fundamental class, and $[\Delta]$ is the Thom class of the diagonal $\Delta$ in $(\sym_g(\Sigma_g))^2$. 
\end{prop}

\begin{proof}
By comparing the non-triviality of   the homology groups $H_*((T^g)^2)$ and $H_*((\sym_g \Sigma_g)^2, (\sym_g \Sigma_g)^2-\Delta)$, dimension $2g$ is the unique dimension on which a homomorphism between them can be non-zero. Note that $H_*((T^g)^2)\cong H_*((\sym_g \Sigma_g)^2, (\sym_g \Sigma_g)^2-\Delta)\cong \mathbb{Z}$. This proposition holds from definition of intersection number.  
\end{proof}

By the construction of a diagram $(\Sigma_g, \alpha, \beta)$, loops $\alpha_i$ and $\beta_j$ are assumed to intersect transvarsally. Thus, we have always that  $I(\bar{\iota}_\alpha, \bar{\iota}_\beta; \mathbf{x})=\pm 1$ for any point in $T_\alpha \cap T_\beta$. Of course, the sign of $I(\bar{\iota}_\alpha, \bar{\iota}_\beta; \mathbf{x})$ depends on the choice of the orientation of the domain $T^g$ and the orientation of the diagonal $\Delta$ in $(\sym_g(\Sigma_g))^2$. The diagonal $\Delta$ has a classical orientation because of complex structure, but the orientation of $T^g$ depends on the choices of orientation of $\alpha_i$'s or $\beta_j$' and their orders. Fortunately, the  $I(\bar{\iota}_\alpha, \bar{\iota}_\beta; \mathbf{x})=I(\bar{\iota}_\alpha, \bar{\iota}_\beta; \mathbf{y})$ does not depend on these choices. 

\begin{theorem}\cite{spinc}
Let $\phi$ be a Whitney disk connecting $\mathbf{x}$ and $\mathbf{y}$. Then its Maslov index 
$\mu(\phi) \equiv 0 \mod 2$ if $I(\bar{\iota}_\alpha, \bar{\iota}_\beta; \mathbf{x})=I(\bar{\iota}_\alpha, \bar{\iota}_\beta; \mathbf{y}) $; $\mu(\phi) \equiv 1 \mod 2$ otherwise, i.e. $I(\bar{\iota}_\alpha, \bar{\iota}_\beta; \mathbf{x})= -I(\bar{\iota}_\alpha, \bar{\iota}_\beta; \mathbf{y}) $.     
\end{theorem}

By \cite[Lemma 7.1]{Whitneydisk} (see the proof of \cite[Theorem 4.9]{spincc}), if we fix the boundary, the set of Maslov indices of all possible  Whitney disks connecting given two intersection is either the set of all even numbers of the set of all odd numbers.

\bibliographystyle{plain}
\bibliography{123.bib}

\begin{thebibliography}{10}

\bibitem{Fukaya}
D.~Auroux.
\newblock A beginner’s introduction to fukaya categories.
\newblock {\em Contact and symplectic topology}, 26, 2014.

\bibitem{Brown}
R.~F. Brown.
\newblock {\em The Lefschetz Fixed Point Theorem}.
\newblock Glenview (Ill.): Scott, Foresman, 1971.

\bibitem{R-Gautschi}
R.~Gautschi.
\newblock Floer homology of algebraically finite mapping classes.
\newblock {\em J. Symplectic Geom.}, pages 715--765, 2003.

\bibitem{Zhao}
Y.~Gu and X.~Zhao.
\newblock Common value pairs and their estimations.
\newblock {\em Bulletin of the Belgian Mathematical Society Simon Stevin}, 24(4):725--739, 2017.

\bibitem{Zhao2022}
Y.~Gu and X.~Zhao.
\newblock Geometric intersections of loops on surfaces.
\newblock {\em Topology and its Applications}, 318:108205, 2022.

\bibitem{Jiang}
B.~Jiang.
\newblock A primer of nielsen fixed point theory.
\newblock In {\em Handbook of Topological Fixed Point Theory}. 2005.

\bibitem{Lipshitz}
R.~Lipshitz.
\newblock A cylindrical reformulation of heegaard floer homology.
\newblock {\em Geom. Topol. 10 (2006), 955–1097}, 2006.

\bibitem{McCord1997}
C.~K. McCord.
\newblock A nielsen theory for intersection numbers.
\newblock {\em Fund. Math.}, 152(2):117--150, 1997.

\bibitem{Ni-fibred-knot}
Y.~Ni.
\newblock Knot floer homology detects fibred knots.
\newblock {\em Invent. Math. 170, 577–608}, 2007.

\bibitem{Ni}
Y.~Ni.
\newblock Link floer homology detects the thurston norm.
\newblock {\em Geom. Topol. 13(5): 2991-3019}, 2009.

\bibitem{OZ-knot-genus}
P.~Ozsv\'{a}th and Z.~Szab\'{o}.
\newblock Holomorphic disks and genus bounds.
\newblock {\em Geom. Topol. 8 (2004), 311–334}, 2004.

\bibitem{spincc}
P.~Ozsv\'{a}th and Z.~Szab\'{o}.
\newblock Holomorphic disks and topological invariants for closed three manifolds.
\newblock {\em Annals of Mathmatics,159,1027-1158}, 2004.

\bibitem{spincc-2}
P.~Ozsv\'{a}th and Z.~Szab\'{o}.
\newblock Holomorphic disks and topological invariants for closed three manifolds.
\newblock {\em Annals of Mathmatics,159, 1159-1245}, 2004.

\bibitem{OZ-Thurston-norm}
P.~Ozsv\'{a}th and Z.~Szab\'{o}.
\newblock Link floer homology and the thurston norm.
\newblock {\em J. Amer. Math. Soc. 21 (2008), 671–709}, 2008.

\bibitem{Whitneydisk}
Zoltán~Szab\'{o} Peter~Ozsv\'{a}th.
\newblock An introduction to heegaard floer homology.
\newblock {\em Floer homology, gauge theory, and low-dimensional topology 5, 3-27}, 2004.

\bibitem{Wang}
S.~Sarkar and J.~Wang.
\newblock An algorithm for computing some heegaard floer homologies.
\newblock {\em Annals of Mathematics, 171, 1213–1236}, 2010.

\bibitem{spinc}
Z.~Szab\'{o}.
\newblock Lecture notes on heegaard floer homology.
\newblock In {\em Low Dimensionl Topology}. American mathmatical Society,Institute for Advanced Study, 2009.

\end{thebibliography}
\end{document}